\documentclass[12pt]{amsart} 
\usepackage{amsmath,amssymb,amscd,amsthm,bm}
\usepackage{latexsym}
\usepackage{graphicx}
\usepackage[english]{babel}
\usepackage[latin1]{inputenc}       
\usepackage{textcomp}
\usepackage[margin=1.4in]{geometry}
\usepackage{times}
\usepackage{pgf,pgfarrows,pgfnodes,pgfautomata,pgfheaps}
\usepackage{colortbl}

\newtheorem{theorem}{Theorem}
\newtheorem{corollary}{Corollary}
\newtheorem{lemma}{Lemma}
\newtheorem{proposition}{Proposition}
\newtheorem{remark}{Remark}

\begin{document}

\title[Poincar\'e inequality for any power]{Improving Beckner's bound via Hermite functions}
\author[Paata Ivanisvili and  Alexander Volberg]{Paata Ivanisvili and Alexander Volberg}
\address{Department of Mathematics, Kent State University, Kent, OH 44240}
\email{ivanishvili.paata@gmail.com}
\address{Department of Mathematics, Michigan State University} 
\email{volberg@math.msu.edu}
\thanks{AV is partially supported by the NSF grant DMS-1600065 and by the Hausdorff Institute for Mathematics, Bonn, Germany}

\begin{abstract}
We obtain an improvement of the Beckner's inequality $\| f\|^{2}_{2} -\|f\|^{2}_{p} \leq (2-p) \| \nabla f\|_{2}^{2}$ valid for $p \in [1,2]$ and the Gaussian measure. Our improvement is essential for the intermediate case $p \in (1,2)$, and moreover, we find the  natural extension of the inequality for any real $p$.
\end{abstract}
\maketitle

\section{Introduction}
\subsection{The history of the problem}
The Poincar\'e inequality~\cite{JN}  for the standard Gaussian measure 
$d\gamma_{n} = \frac{e^{-|x|^{2}/2}}{\sqrt{(2\pi)^{n}}}dx$  states that 
\begin{align}\label{poincare}
\int_{\mathbb{R}^{n}} f^{2} d\gamma_{n} - \left(\int_{\mathbb{R}^{n}} f d\gamma_{n}\right)^{2} \leq \int_{\mathbb{R}^{n}} | \nabla f |^{2} d\gamma_{n}
\end{align}
for any smooth bounded function $f :\mathbb{R}^{n} \to \mathbb{R}$.  Later William Beckner \cite{WB} generalized (\ref{poincare}) for any real power $p$, $1 \leq p \leq 2$  as follows 
\begin{align}\label{beckner}
\int_{\mathbb{R}^{n}} f^{p} d\gamma_{n} - \left(\int_{\mathbb{R}^{n}} f d\gamma_{n}\right)^{p} \leq \frac{p(p-1)}{2}\int_{\mathbb{R}^{n}} f^{p-2}| \nabla f |^{2} d\gamma_{n}
\end{align}
for any smooth bounded $f : \mathbb{R}^{n} \to (0,\infty)$. We caution the reader that in \cite{WB} inequality (\ref{beckner}) was formulated in a slightly different but equivalent form (see Theorem~1, inequality (3) in \cite{WB}). It should be also mentioned that in case $p=2$ inequality  (\ref{beckner}) does coincide with  (\ref{poincare}) for all $f \geq 0$  but it does not imply the Poincar\'e inequality for the functions taking the negative values, especially when $\int_{\mathbb{R}^{n}} f d \gamma_{n}=0$. If $p \to 1+$ then (\ref{beckner}) provides us with log-Sobolev inequality (see \cite{WB}). In general, the constant $\frac{p(p-1)}{2}$ is sharp in the right hand side of (\ref{beckner}) as it can be seen for $n=1$ on the test functions $f(x)=e^{\varepsilon x}$ by sending $\varepsilon \to 0$. 

Later Beckner's inequality (\ref{beckner}) was studied by many mathematicians for different measures, in different settings and for different spaces as well. For possible references we refer the reader to \cite{ABD, ALS,  BCR1, BCR2, BR1, Bob1, Bob2, BBL, Chaf, WFY, RK, KO}. 

An analysis done in \cite{IV} indicates that the right hand side (RHS) of (\ref{beckner}) can be improved. In the present paper we address this issue: what is the precise estimate of the difference  given in the left hand side (LHS) of (\ref{beckner}), and whether the requirement $p \in [1,2]$ can be avoided by slightly changing  the RHS of (\ref{beckner}). 

We give complete answers to these questions.
For example, if $p=\frac{3}{2}$ we will obtain  an  improvement in Beckner's inequality (\ref{beckner})
\begin{align}
&\int_{\mathbb{R}^{n}} f^{3/2} d\gamma_{n} - \left(\int_{\mathbb{R}^{n}} f d\gamma_{n}\right)^{3/2} \leq \label{b3/2}\\
 &\int_{\mathbb{R}^{n}}\left(   f^{3/2} - \frac{1}{\sqrt{2}}(2f-\sqrt{f^{2}+ |\nabla f|^{2}})\sqrt{f+\sqrt{f^{2}+|\nabla f|^{2}}} \right)d\gamma_{n}. \nonumber
\end{align}
The LHS of  (\ref{b3/2}) coincides with the LHS of (\ref{beckner}) for $p=3/2$, but the RHS of (\ref{b3/2}) is strictly smaller than the RHS in (\ref{beckner}). Indeed, notice that we have the following \emph{pointwise} inequality
\begin{align}\label{impr1}
x^{3/2} - \frac{1}{\sqrt{2}}(2x-\sqrt{x^{2}+ y^{2}})\sqrt{x+\sqrt{x^{2}+y^{2}}} \leq \frac{3}{8} x^{-1/2}y^{2} \quad \text{for all} \quad  x, y \geq 0,
\end{align}
which follows from the homogeneity, i.e., take $x=1$. As one can see the improvement of Beckner's inequality (\ref{beckner}) is essential. Indeed, if $y \to \infty$ then the RHS of (\ref{impr1}) increases as $y^{2}$ whereas the LHS of (\ref{impr1}) increases as $y^{3/2}$. Also notice that if $x \to 0$ then the difference in (\ref{impr1}) tends to infinity. The only place where the quantities in (\ref{impr1}) are comparable is when $y/x \to 0$. 

\subsection{Main results}
Let $k$ be a real parameter. 
Let  $H_{k}(x)$ be the Hermite function  such that it satisfies the Hermite differential equation
\begin{align}\label{hermite}
H_{k}''-xH_{k}'+k H_{k}=0,  \quad x \in \mathbb{R},
\end{align}
and which grows relatively slowly $H_{k}(x) = x^{k}+o(x^{k})$ as $x \to +\infty$. If $k$ is a nonnegative integer then $H_{k}$ is the \emph{probabilists'} Hermite polynomial of degree $k$ with the leading coefficient $1$, for example, $H_{0}(x)=1, H_{1}(x)=x, H_{2}(x)=x^{2}-1$ etc.  In general, for arbitrary $k \in \mathbb{R}$ one should think that $H_{k}$ is the analytic extension of the Hermite polynomials in $k$ (existence and many other properties  will be mentioned in Section~\ref{prf}).

For $k \in \mathbb{R}$ let $R_{k}$ be the rightmost zero of $H_{k}(x)$ (see Lemma~\ref{EMMlemma}). If $k \leq 0$ then we set $R_{k}=-\infty$. 
 Define $F_{k}(x)$ as follows  
\begin{align}\label{bdef}
F_{k}\left(\left| \frac{H'_{k}(q)}{H_{k}(q)} \right|\right) = \frac{H_{k+1}(q)}{H^{1+\frac{1}{k}}_{k}(q)} \quad  \text{for} \quad q \in (R_{k}, \infty).
\end{align}


We will see in the next section $F_{k} \in C^{2}([0,\infty))$ is well-defined and $F_{k}(0)=1$. Moreover,  if $k > -1$ then $F_{k}$ will be decreasing concave function, and  if $k<-1$ then $F_{k}$ will be increasing convex function. 


One may observe that
\begin{align*}
F_{1}(y)=1-y^{2}; \quad F_{2}(y)=\frac{1}{\sqrt{2}}(2-\sqrt{1+ y^{2}})\sqrt{1+\sqrt{1+y^{2}}}. 
\end{align*}

If $k=0$ then  definition (\ref{bdef}) should be understood in the limiting sense as follows
\begin{align*}
F_{\exp}(H_{-1}(q)) = q\exp\left( \alpha - \int_{1}^{q} H_{-1}(s)ds\right) \quad \text{for all} \quad q \in \mathbb{R}, 
\end{align*}
where 
\begin{align}\label{kap}
\alpha = \int_{1}^{\infty}\left(H_{-1}(s)-\frac{1}{s}\right)ds \approx  - 0.266\ldots . 
\end{align}

\begin{theorem}\label{better}
For any $p \in \mathbb{R} \setminus [0,1]$ and any  smooth bounded $f\geq 0$ with $\int_{\mathbb{R}^{n}} f^{p} d\gamma_{n}<\infty$   we have 
\begin{align}\label{our}
\int_{\mathbb{R}^{n}}f^{p} F_{\frac{1}{p-1}}\left(\frac{|\nabla f|}{f}\right) d\gamma_{n} \leq  \left( \int_{\mathbb{R}^{n}} f d\gamma_{n} \right)^{p}.
\end{align}
The inequality is reversed if $p \in (0,1)$. 
\end{theorem} 
The theorem  improves  Beckner's inequality (\ref{beckner}). This  will follow by taking the first two nonzero Taylor terms of $F_{\frac{1}{p-1}}(t)$ as its lower estimate.
\begin{proposition}\label{prob}
We have pointwise improvement in  Beckner's inequality (\ref{beckner}), i.e., 
\begin{align}\label{bbeckner}
1-\frac{p(p-1)}{2}t^{2} \leq  F_{\frac{1}{p-1}}\left(t\right)    \quad \text{for all} \quad t \geq 0, \; p \in (1,2].
\end{align}
\end{proposition}

The improvement will be essential when $t \to \infty$. For example, it will become clear in the next section that as $t \to \infty$ we have 
\begin{align}
&F_{\frac{1}{p-1}}(t) \sim -t^{p} \left( H'_{\frac{1}{p-1}}(R_{\frac{1}{p-1}})\right)^{1-p} \quad \text{for} \quad p>1; \label{a1}\\
&F_{\frac{1}{p-1}}(t) \sim \left( \frac{p}{1-p}\right) \left(\frac{e^{t^{2}/2} \, \sqrt{2 \pi}}{t \Gamma(\frac{1}{1-p}) }\right)^{1-p} \quad \text{for} \quad p<1, \quad p \neq 0. \label{a2}
\end{align}

Our theorem interpolates several inequalities. If $p\to 1+$ then (\ref{our}) gives log-Sobolev inequality. If $p=2$ then (\ref{our}) provides us with Poincar\'e inequality. If $p\to \pm \infty$ then we obtain \emph{$e$-Sobolev inequality}: 
\begin{corollary}\label{e-sob}
For any smooth bounded $f$ we have 
\begin{align*}
\int_{\mathbb{R}^{n}} \exp(f)\,  F_{\exp}(|\nabla f|) d\gamma_{n} \leq \exp\left(  \int_{\mathbb{R}^{n}} f d\gamma_{n}\right). 
\end{align*}
\end{corollary}

Finally if $p \to 0$ we obtain \emph{negative log-Sobolev} inequality:
\begin{corollary}\label{nlog}
For any smooth bounded $f\geq 0$ with $\int_{\mathbb{R}^{n}} \ln f d\gamma_{n}>-\infty$   we have 
\begin{align*}
\int_{\mathbb{R}^{n}} - \ln f d\gamma_{n}  + \ln \left(\int_{\mathbb{R}^{n}} f d\gamma_{n} \right) \leq  \int_{\mathbb{R}^{n}}- F_{-\ln}\left(\frac{|\nabla f|}{f}\right)d\gamma_{n}
\end{align*}
where $F_{-\ln}(t)$ is defined as follows 
\begin{align*}
F_{-\ln} \left( \frac{H_{-2}(x)}{H_{-1}(x)}\right) = \int_{1}^{x}H_{-1}(s)ds-c + \ln H_{-1}(x), \quad x \in \mathbb{R}.
\end{align*}
\end{corollary}

 It is worth mentioning that the current paper provides with  estimates of $\Phi$-entropy (see \cite{Chaf}):
 \begin{align*}
 \boldsymbol{\mathrm{Ent}}_{\gamma_{n}}^{\Phi}(f) \stackrel{\mathrm{def}}{=} \int_{\mathbb{R}^{n}}\Phi(f)d\gamma_{n} - \Phi\left(\int_{\mathbb{R}^{n}} f d\gamma_{n} \right)
 \end{align*}
 for the following fundamental examples:
 \begin{align*}
  &\Phi(x) = x^{p} \quad \text{for}  \quad p \in \mathbb{R}\setminus [0,1] \quad \text{Theorem~\ref{better}};\\
  &\Phi(x) = -x^{p} \quad \text{for}  \quad p \in (0,1) \quad \text{Theorem~\ref{better}};\\
  &\Phi(x) = e^{x},  \quad \text{Corollary~\ref{e-sob}, or $p \to \pm \infty$ in Theorem~\ref{better}};\\
  &\Phi(x) = -\ln x,  \quad \text{Corollary~\ref{nlog}, or $p \to 0$ in Theorem~\ref{better}};\\
  &\Phi(x) = x\ln x,  \quad \text{$p \to 1$ in Theorem~\ref{better}}.
 \end{align*}
\section{The proof of the theorem}\label{prf}
The proof of the theorem amounts to check that the  real valued function 
\begin{align}\label{f1}
M(x,y)=x^{p}F_{k}\left( \frac{y}{x} \right)
\end{align}
 defined on $[\varepsilon,\infty)\times [0, \infty)$ for any $\varepsilon>0$ obeys necessary smoothness condition, it has a boundary condition $M(x,0)=x^{p}$ and it  satisfies the following partial differential inequality
\begin{align}\label{matrica}
\begin{pmatrix}
M_{xx}+\frac{M_{y}}{y}  & M_{xy}\\
M_{xy} & M_{yy}
\end{pmatrix} \leq 0,
\end{align}
with  reversed inequality in (\ref{matrica}) if $p\in (0,1)$. Then by Theorem~1 in \cite{IV} we obtain that 
\begin{align*}
\int_{\mathbb{R}^{n}} f^{p}F_{k}\left( \frac{|\nabla f|}{f}\right) d\gamma_{n} = \int_{\mathbb{R}^{n}} M(f, |\nabla f| )d\gamma_{n} \leq M\left(\int_{\mathbb{R}^{n}}f d\gamma_{n},0 \right) = \left( \int_{\mathbb{R}^{n}} fd\gamma_{n}\right)^{p} 
\end{align*}
for any smooth bounded $f \geq \varepsilon$ which is the statement of the theorem we want to prove (except we need to justify the passage to the limit $\varepsilon \to 0$ and this will be done later). Notice that the inequality is reversed if $p \in (0,1)$, indeed, in this case we should work with  $-M(x,y)$ instead of $M(x,y)$. 


Next we will need some tools regarding the Hermite functions $H_{k}$. 

\subsection{Properties of Hermite functions}

$H_{k}$ can be defined  (see \cite{HO}) by 
\begin{align}\label{hyp1}
H_{k}(x) =  -\frac{2^{-k/2} \sin (\pi k)\;   \Gamma(k+1)}{2\pi} \sum_{n=0}^{\infty} \frac{\Gamma((n-k)/2)}{n!} (-x \sqrt{2} )^{n},
\end{align}
or in terms of  the confluent hypergeometric functions (see \cite{LD}) by
\begin{align}
H_{k}(x) =  &\sqrt{\frac{2^{k}}{\pi}}\left[\cos\left( \frac{\pi k}{2}\right) \, \Gamma\left( \frac{k+1}{2}\right)\,  {}_{1}F_{1}\left(-\frac{k}{2}, \frac{1}{2}; \frac{x^{2}}{2}\right)\right. \label{hyp2} \\
&+\left. t\sqrt{2} \sin \left( \frac{\pi k}{2} \right) \, \Gamma\left( \frac{k}{2}+1\right)\,  {}_{1}F_{1}\left(\frac{1-k}{2}, \frac{3}{2}; \frac{x^{2}}{2}\right) \right].\nonumber
\end{align}
If $k$ is a nonnegative integer then one should understand (\ref{hyp1}) and (\ref{hyp2}) in the limiting sense.   Notice the following recurrence properties:\begin{align}
&H'_{k}(x) = k H_{k-1}(x); \label{pirveli}\\
&H_{k+1}(x) = x H_{k}(x) - H'_{k}(x). \label{meore}
\end{align}
These properties follow from (\ref{hyp1}) and the fact that $\Gamma(z+1)=z\Gamma(z)$. 

We also notice that  
  \begin{align*}
H_{k}(x) := e^{x^{2}/4}D_{k}(x),
\end{align*}
 where $D_{k}(x)$ is the \emph{parabolic cylinder function}, i.e., it is the solution of the equation 
\begin{align*}
D''_{k}+\left( k+\frac{1}{2}-\frac{x^{2}}{4}\right)D_{k}=0. 
\end{align*}
Since $H_{k}(x)$ is an entire function in $x$ and $k$ (see \cite{temme} for the parabolic cylinder function) sometimes it will be convenient to write $H(x,k)$ instead of $H_{k}(x)$.  
The precise asymptotic for $x \to +\infty$, $x>0$  and any $k \in \mathbb{R}$ is given as follows 
\begin{align}\label{loran}
H_{k}(x) \sim x^{k} \cdot  \sum_{n=0}^{\infty} (-1)^{n} \frac{(-k)_{2n}}{n! (2x^{2})^{n}}.
\end{align}
 Here $(a)_{n}=1$ if $n=0$ and $(a)_{n}=a(a+1)\ldots(a+n-1)$ if $n>0$.   When $x \to -\infty$  we have 
 \begin{align}\label{loran-}
 H_{k}(x) \sim |x|^{k}\cos(k \pi) \sum_{n=0}^{\infty} (-1)^{n} \frac{(-k)_{2n}}{n! (2x^{2})^{n}}  + \frac{\sqrt{2 \pi}}{\Gamma(-k)} |x|^{-k-1}e^{x^{2}/2}\sum_{n=0}^{\infty} \frac{(1+k)_{2n}}{n! (2x^{2})^{n}}.
 \end{align}
 We refer the reader to \cite{temme, NIST}. For instance, for (\ref{loran}) we can use the asymptotic formula (12.9.1) in \cite{NIST}  for the parabolic cylinder function. To verify (\ref{loran-}) we can express $H_{k}(-x)$ as a  linear combination of two parabolic cylinder functions but having argument $x$ instead of $-x$ (see (12.2.15) in \cite{NIST}), and then we can use (12.9.1) and (12.9.2) in \cite{NIST}. 

Next we will need the result of  Elbert--Muldoon \cite{EMM1}  which describes the behavior of the real zeros of $H_{k}(x)$ for any real $k$. 
 
 \begin{lemma}\label{EMMlemma}
 For $k\leq 0$,  $H_{k}(x)$ has no real zeros, and it is positive on the real axis. For $n<k \leq n+1$, $n=0,1,\ldots, $ $H_{k}(x)$ has $n+1$ real  zeros. Each zero is increasing function of  $k$ on its interval of definition.
 \end{lemma}
 The proof of the lemma is Theorem~3.1 in \cite{EMM1}.   It is explained in the paper that  as $k$ passes through each nonnegative integer $n$ a new leftmost zero appears at $-\infty$ while the right-most zero passes through the largest zero of $H_{k}(x)$. More precise information about the asymptotic  behavior of the zeros as $k \to \infty$ can be found in \cite{EMM2}. 
 
 Further we will need Tur\'an's inequality for $H_{k}$ for any real $k$. 
 \begin{lemma}\label{turl}
 We have the following Tur\'an's inequality:  
 \begin{align}
 H_{k}^{2}(x)-H_{k-1}(x)H_{k+1}(x)> 0 \quad \text{for all} \quad k \in \mathbb{R}, \; x \geq  L_{k} \label{tu1}
 \end{align}
 where $L_{k}$ denotes the leftmost zero of $H_{k}$. If $k\leq 0$ then $L_{k}=-\infty$. 
 \end{lemma}
 The lemma is known as Tur\'an's inequality when $k$ is a nonnegative integer. Unfortunately we could not find the reference in the case when $k$ is different from a positive integer therefore we decided to include the proof of the lemma. 
 
  The following  is  borrowed from  \cite{MT}. 
\begin{proof}
Take $f(x) = e^{-\frac{x^{2}}{2}}(H_{k}^{2}(x)-H_{k-1}(x)H_{k+1}(x))$. Asymptotic formulas (\ref{loran}) and (\ref{loran-}) imply that 
\begin{align}
&\lim_{x \to +\infty} f(x) =0 \quad \text{for all} \quad k \in \mathbb{R};\nonumber\\
&f(x) \sim \sqrt{2 \pi} |x| >0\quad  \text{for} \quad x \to -\infty, \quad k =0;\nonumber\\
&f(x) \sim \frac{2\pi e^{x^{2}/2}}{\Gamma(-k)\Gamma(-k+1)}|x|^{-2k-2} \quad \text{for} \quad x \to -\infty, \quad k \notin \{0\}\cup \mathbb{N}. \label{atinf}
\end{align}
On the other hand notice that 
\begin{align}\label{der}
f'(x) = -e^{-\frac{x^{2}}{2}} H_{k}H_{k-1}.
\end{align}
If $k\leq 0$ then  by Lemma~\ref{EMMlemma} $f'<0$,  and because of the conditions $f(-\infty)=+\infty$ and $f(\infty)=0$ we obtain that $f>0$ on $\mathbb{R}$.
To verify the statement for $k>0$ we notice that 
\begin{align}\label{der1}
f''(x) = e^{-\frac{x^{2}}{2}}(H_{k}^{2} -k H_{k-1}^{2}).
\end{align}
Now we notice that if $H_{k}(c)=0$ then $H_{k-1}(c)\neq 0$. Indeed, assume contrary $H_{k-1}(c)=0$.  Then by (\ref{pirveli}) we have  $H'_{k}(c)=0$ and by (\ref{hermite}) we obtain $H''_{k}(c)=0$,  and again taking derivative in (\ref{pirveli}) we obtain that $H_{k-2}(c)=0$. Repeating this process we obtain that $H_{k-N}(c)=0$ for any large integer $N>0$. But this contradicts to Lemma~\ref{EMMlemma}. 

Thus by (\ref{der}) and (\ref{der1}) we obtain that $c$ is a point of the local minimum of $f$ if and only if $ H_{k-1}(c)=0$. Then $f(c) = e^{-x^{2}/2}H_{k}^{2}(c)>0$.
Finally we obtain that $f : [L_{k}, \infty) \to \mathbb{R}$ is positive on its local minimum points, $f(\infty)=0$ and $f(L_{k}) >0$ (because $H_{k-1}, H_{k+1}$ have opposite signs at zeros of $H_{k}$ by (\ref{meore})). Therefore $f > 0$ on $[L_{k}, \infty) \to \mathbb{R}$ and the lemma is proved. 

 \end{proof}
\begin{remark}
If $k \in \mathbb{N}$  then  $H_{k}$ is the probabilists' Hermite polynomial of degree $k$, so $f(x)$ will be  even and inequality (\ref{tu1}) will hold for all $x \in \mathbb{R}$ which confirms the classical Tur\'an's inequality.   
 However, if $k>0$ but $k \notin \mathbb{N}$  then (\ref{tu1}) fails when $x \to -\infty$ (see (\ref{atinf})).  
 \end{remark}

Finally the next corollary together with Lemma~\ref{EMMlemma} implies that $\left|\frac{H'_{k}}{H_{k}}\right|=\mathrm{sign}(k) \frac{H'_{k}(q)}{H_{k}(q)} $ is positive and decreasing for $q \in (R_{k}, \infty)$ and $k\in \mathbb{R}\setminus\{0\}$. 
\begin{corollary}\label{cor}
For any $x \ge L_{k}$ and any $k \in \mathbb{R}\setminus\{0\}$ we have 
\begin{align*}
\mathrm{sign}[(H'_{k})^{2}-H_{k}H''_{k}]=\mathrm{sign}(k).
\end{align*}
\end{corollary} 
\begin{proof}
The proof follows from Lemma~\ref{turl} and the following identity
\begin{align}\label{identity1}
k(H_{k}^{2}-H_{k-1}H_{k+1}) = (H'_{k})^{2}-H_{k}H''_{k}
\end{align}
from (\ref{hermite}), (\ref{pirveli}) and (\ref{meore}). 
\end{proof}
\subsection{Checking the partial differential inequality}
Let $p=1+\frac{1}{k}$.   Further we assume $k \neq 0, -1$. Define $F=F_{k}$ as in the introduction:
\begin{align}\label{opr}
F(t) = \frac{H_{k+1}(q)}{H_{k}^{1+1/k}(q)} \quad \text{where} \quad \left| \frac{H'_{k}(q)}{H_{k}(q)}\right|=t, \quad q \in (R_{k}, \infty), \quad t \in (0, \infty).
\end{align}

Notice that by Corollary~\ref{cor} function  $\left| \frac{H'_{k}(q)}{H_{k}(q)}\right|=\mathrm{sign}(k) \frac{H'_{k}(q)}{H_{k}(q)}$ is positive decreasing in $q$ for $q \in (R_{k}, \infty)$, moreover by (\ref{loran}) we have 
$\frac{H'_{k}(q)}{H_{k}(q)} \sim \frac{k}{q}$ when $q \to +\infty$. From the same asymptotic formulas it follows that when $t \to 0+$ we have 
\begin{align*}
F(t)= 1-\frac{p(p-1)}{2}\, t^{2} +O(t^{4}).
\end{align*}
Therefore $F$ is well-defined function and $F \in C^{2}([0,\infty))$. 

Take  a positive $\varepsilon>0$ and define $M(x,y)$ as in (\ref{f1}):
\begin{align}\label{k1}
M(x,y) :=x^{p} F\left(\frac{y}{x}\right) \quad \text{for} \quad  y\geq 0, \quad x > \varepsilon >0.
\end{align}
Clearly $M(x,\sqrt{y}) \in C^{2}([\varepsilon, \infty)\times[0, \infty))$. By Theorem~1 in \cite{IV} we have inequality 
\begin{align}\label{ivo1}
\int_{\mathbb{R}^{n}} M(f,|\nabla f|) d\gamma_{n} \leq M\left(\int_{\mathbb{R}^{n}}f d\gamma_{n}, 0 \right)
\end{align}
for all smooth bounded $f \geq \varepsilon$ if  (\ref{matrica}) holds. In terms of $F$ (see (\ref{k1})) condition (\ref{matrica}) takes the form 
\begin{align}
&tFF''p(p-1)+F'F''-t(p-1)^{2}(F')^{2} \geq 0  \qquad \text{i.e., the determinant of (\ref{matrica}) is nonnegative}\label{determinant}\\
& F'' (t+t^{3})+F'(2t^{2}+1-2pt^{2})+Fp(p-1)t\leq 0 \qquad \text{i.e., the trace of (\ref{matrica}) is nonpositive} \label{trace}
\end{align}
where $t = \frac{y}{x}$ is the argument of $F$.  
In fact we will show that we have equality in (\ref{determinant}) instead of inequality therefore the sign of (\ref{matrica}) will depend on the sign of trace (\ref{trace}).
We will see that inequality (\ref{trace}) will be reversed for $p \in (0,1)$. 
 
 From (\ref{opr}), (\ref{identity1}) and (\ref{tu1}) we obtain 
\begin{align}
 &F'(t) =-\frac{k+1}{|k|} \frac{1}{H_{k}^{1/k}};\label{odin}\\
 &F''(t) = \frac{F'}{|k|}\cdot   \frac{H_{k}H_{k-1}}{H_{k}^{2}-H_{k+1}H_{k-1}};\label{dva}\\
 &F(t) = -\frac{|k|}{k+1} \frac{H_{k+1}}{H_{k}}\, F'. \label{tri}
\end{align}

If we plug (\ref{dva}) and (\ref{tri}) into (\ref{determinant}) we obtain that the left hand side of (\ref{determinant}) is zero.  If we plug (\ref{dva}) and (\ref{tri}) into (\ref{trace}) we obtain 
\begin{align*}
\text{LHS of}\; (\ref{trace}) = \left[ \frac{(kH_{k-1}^{2}-H_{k}^{2}+H_{k-1}H_{k+1})^{2}+H_{k-1}^{2}H_{k}^{2}}{H_{k}^{2}(H_{k}^{2}-H_{k+1}H_{k-1})}\right]\, F'. 
\end{align*}
Thus the sign of LHS of (\ref{trace}) coincides with the sign of $F'$ which coincides with $\mathrm{sign}(-(k+1))$. 
The condition  $p \in \mathbb{R}\setminus [0,1]$ implies that $k>-1$ and therefore (\ref{matrica}) holds. The condition $p \in (0,1)$ implies that $k<-1$ and therefore inequality in (\ref{matrica}) is reversed.

Thus we have obtained (\ref{ivo1}) for smooth bounded functions $f \geq \varepsilon$. Next we claim that for an arbitrary smooth bounded $f\geq 0$ with $\int_{\mathbb{R}^{n}}f^{p} d\gamma_{n}<\infty$ we can apply the inequality to $f_{\varepsilon}:= f +\varepsilon$ and  send $\varepsilon$ to $0$ in (\ref{our}). Indeed, it follows from (\ref{bdef}) and (\ref{loran}) that as $t \to \infty$ we have
\begin{align*}
&F(t) \sim -t^{1+\frac{1}{k}} (H'_{k}(R_{k}))^{-\frac{1}{k}} \quad \text{for} \quad k>0;\\
&F(t) \sim \mathrm{sign}(-1-k) \left(\frac{e^{t^{2}/2} \, \sqrt{2 \pi}}{t \Gamma(-1-k) }\right)^{-\frac{1}{k}} |1+k|^{1+\frac{1}{k}}\quad \text{for} \quad k<0, \quad k \neq -1. 
\end{align*}
Thus for $p>1$ (i.e., $k>0$) the claim about the limit follows from the estimate $|F(t)| \leq C_{1}+ C_{2} t^{p}$ together with  the Lebesgue dominated convergence theorem.

  If   $p <0$ (i.e., $k \in (-1,0)$) we rewrite (\ref{our}) in a standard way as follows 
\begin{align}\label{lim}
\int_{\mathbb{R}^{n}} f_{\varepsilon}^{p} d\gamma_{n} - \left( \int_{\mathbb{R}^{n}} f_{\varepsilon} d\gamma_{n} \right)^{p} \leq \int_{\mathbb{R}^{n}}f_{\varepsilon}^{p}\left(1-
F\left(\frac{|\nabla f |}{f_{\varepsilon}}\right)\right) d\gamma_{n}.
\end{align}
Since $f$ is bounded, $f\geq 0$  and  $\int_{\mathbb{R}^{n}} f^{p} d\gamma_{n} <\infty$ there is no issue with the left hand side of (\ref{lim}) when $\varepsilon \to 0$. For the right hand side of (\ref{lim}) we notice that the function $x^{p}(1-F(y/x))$ is nonnegative and decreasing in $x$ then the claim follows from the monotone convergence theorem. The non negativity follows from the observation that $F(0)=1$ and $F'<0$ (see (\ref{odin}) where we have  $k>-1$).
The monotonicity follows from (\ref{bdef}), (\ref{odin}), (\ref{pirveli}) and  the straightforward computations
\begin{align}\label{hh4}
\frac{\partial }{\partial x} \left(x^{p}(1-F(y/x)) \right)=x^{p-1}\left(p-pF(t)+tF'(t) \right)=x^{p-1} p \left[1-\frac{q}{H_{k}^{\frac{1}{k}}(q)} \right],
\end{align} 
where $|k| \frac{H_{k-1}(q)}{H_{k}(q)}=t=\frac{y}{x}$ and $q\in (R_{k}, \infty)$.  The last expression in (\ref{hh4}) is negative because 
\begin{align*}
1 \geq F(t) = \frac{H_{k+1}}{H_{k}^{1+\frac{1}{k}}}=\frac{qH_{k}-kH_{k-1}}{H_{k}^{1+\frac{1}{k}}} >\frac{q}{H_{k}^{\frac{1}{k}}}. 
\end{align*}

Finally if $p \in (0,1)$ (i.e., $k <-1$) we have the opposite inequality in (\ref{lim}). In this case  the situation is absolutely the same as for $k \in (-1,0)$  except now we should consider  the function $x^{p}(F(y/x)-1)$ which is nonnegative and decreasing in $x$ (see (\ref{hh4})). This finishes the proof of the theorem.


Now let us show  Proposition~\ref{prob}. 
Since $F(0)=1$ it is enough to show a stronger inequality, namely $F'+p(p-1)t \geq 0$. From (\ref{odin})  and the fact that $k\geq 1$ (since $p \in [1,2]$) we obtain that it is enough to show the following inequality 
 \begin{align*}
 -\frac{p}{H_{k}^{1/k}}+p(p-1)\frac{H'_{k}}{H_{k}}\geq 0 \quad \text{for all} \quad  k \geq 1, \; q \in (R_{k}, \infty).
 \end{align*}
Using (\ref{pirveli}) and $p=1+\frac{1}{k}$ we notice that   the inequality can be rewritten as follows $1>\frac{H_{k}(q)}{H_{k-1}^{\frac{k}{k-1}}(q)}$ for all $q \in (R_{k}, \infty)$. To verify the last inequality we remind  that  $F(0)=1$ and $F'(t) <0$. Therefore $F(t)\leq 1$. We recall the definition of $F(t)$ (see (\ref{opr})). It follows that $1\geq F = \frac{H_{k+1}}{H_{k}^{1+1/k}}$ for all $k>0$. The last inequality is the same as 
\begin{align*}
1>\frac{H_{k}(q)}{H_{k-1}^{\frac{k}{k-1}}(q)} \quad \text{for all} \quad q \in (R_{k}, \infty), \quad k \geq 1.
\end{align*}
This finishes the proof of the theorem. 

\subsection{Proof of Corollary~\ref{e-sob} and Corollary~\ref{nlog}:}
Notice that as $t \to 0$ we have  
\begin{align*}
F_{\exp}(y) = 1-\frac{y^{2}}{2}+O(y^{4}) \quad \text{and} \quad F_{-\ln}(y) = -\frac{y^{2}}{2}+O(y^{4}). 
\end{align*}
There are two ways to obtain the corollaries. 

\subsubsection{The first way:} One can check that  
\begin{align*}
&M_{\exp}(x,y)=e^{x}F_{\exp}(y), \quad M_{\exp}(x,0)=e^{x}, \quad M_{\exp}(x,\sqrt{y})\in C^{2}(\mathbb{R}\times \mathbb{R}_{+}); \\
&M_{-\ln}(x,y)=-\ln(x)+F_{-\ln}\left(\frac{y}{x}\right), \quad M_{-\ln }(x,0)=-\ln x, \quad x>0, 
\end{align*}
and $M_{-\ln}(x,\sqrt{y}) \in C^{2}([\varepsilon, \infty) \times \mathbb{R}^{+})$ for any $\varepsilon>0$. 
By  straightforward  computations we notice that if we set $\psi(q) = \alpha - \int_{1}^{q} H_{-1}(s)ds$  then using the identity  $1=qH_{-1}(q)+H_{-2}(q)$ we obtain 
\begin{align*}
F_{\exp}(H_{-1}) = q e^{\psi}, \quad F'_{\exp}(H_{-1})=-e^{\psi} \quad \text{and}  \quad F''_{\exp}(H_{-1}) = - \frac{H_{-1}}{H_{-2}}.
\end{align*}
Similarly we compute that 
\begin{align*}
F'_{-\ln}\left(\frac{H_{-2}}{H_{-1}}\right)=-H_{-1}\quad \text{and} \quad F''_{-\ln}\left(\frac{H_{-2}}{H_{-1}}\right)=- \frac{H_{-2}H_{-1}^{2}}{H_{-1}^{2}-H_{-2}}.
\end{align*}
Next one notices that  $M_{\exp}$ and $M_{-\ln}$  satisfy (\ref{matrica}) (in fact the determinant of (\ref{matrica}) is zero). Then by Theorem~1 in \cite{IV} we obtain the  corollaries.  
The passage to the limit for $M_{-\ln}(x,y)$ when $\varepsilon \to 0$ follows from the monotone convergence theorem. Indeed, we notice that  $-F_{-\ln}(y/x) \geq 0$ is decreasing in $x$. We apply Corollary~\ref{nlog} to $f_{\varepsilon}=f+\varepsilon$ and send $\varepsilon \to 0$.  
\subsubsection{The second way:} We will obtain the corollaries as a limiting case of Theorem~\ref{better}. Indeed, to verify Corollary~\ref{e-sob} let $f^{p}=e^{g}$  in (\ref{our}). Then (\ref{our}) takes the form 
\begin{align}\label{ee-sob}
\int_{\mathbb{R}^{n}}e^{g} F_{\frac{1}{p-1}}\left( \frac{|\nabla g|}{p}\right)d\gamma_{n} \leq \left( \int_{\mathbb{R}^{n}} e^{g/p} d\gamma_{n}\right)^{p}.
\end{align}
Now we take $p \to \infty$. The RHS of (\ref{ee-sob}) tends to $\exp(\int_{\mathbb{R}^{n}}  g d\gamma_{n})$. For the  LHS of (\ref{ee-sob}) we should   compute the limit 
\begin{align*}
F_{\exp}(t) :=\lim_{p \to \infty} F_{\frac{1}{p-1}}\left( \frac{t}{p}\right) = \lim_{p \to \infty} F_{\frac{1}{p-1}}\left( \frac{t}{p-1}\right) = \lim_{k \to 0+} F_{k}(tk). 
\end{align*}
It is clear that $F_{\exp}(0)=1$. 
Next  if we take $k \to 0+$ in  (\ref{bdef}) we obtain 
\begin{align*}
\lim_{k \to 0+} F_{k}\left(\left| \frac{H'_{k}}{H_{k}} \right| \right) = \lim_{k \to 0+} F_{k}\left(k \frac{H_{k-1}}{H_{k}}  \right) = \lim_{k \to 0+} F_{k}\left(k \frac{H_{-1}}{H_{0}}  \right) = F_{\exp}(H_{-1})
\end{align*}

On the other hand for  the RHS of (\ref{bdef}) we have 
\begin{align*}
\lim_{k \to 0+}  \frac{H_{k+1}(q)}{H_{k}^{1+\frac{1}{k}}} = q \lim_{k \to 0+} H_{k}^{-1/k}.
\end{align*}
Here we have used $H_{0}(q)=1$ and $H_{1}(q)=q$. 
Thus it remains to find $\lim_{k \to 0+}H_{k}^{-1/k}$. Notice that $H(x,k):=H_{k}(x)$ is an entire function in $x$ and $k$ (see \cite{temme} for the Parabolic cylinder function).  If we take derivative in $k$ of (\ref{pirveli}) we obtain $H_{xk}(x,k)=H(x,k-1)+kH_{k}(x,k)$ (here subindices denote partial derivatives). Now taking $k=0$ we obtain $H_{xk}(x,0)=H(x,-1)$. Thus $H_{k}(x,0)$ is an antiderivative of $H(x,-1)=H_{-1}$. So 
\begin{align*}
\lim_{k \to 0+} H_{k}^{-1/k}= \lim_{k \to 0+} \exp\left( -\frac{1}{k}\ln (1+kH_{k}(x,0)+o(k))\right) = \exp\left(- \int H_{-1}(s) ds \right).
\end{align*} 
Finally we obtain
\begin{align}\label{hh1}
F_{\exp}(H_{-1}(q)) = q \exp \left(C - \int_{1}^{q}H_{-1} \right)
\end{align}

In order to satisfy the condition $F_{\exp}(0)=1$ the constant $c$ must be chosen as follows $C= \int_{1}^{\infty}(H_{-1}-\frac{1}{s})ds$  (indeed send $q \to \infty$ in (\ref{hh1})). This finishes the proof of Corollary~\ref{e-sob}. It is worth mentioning that we have also obtained  (see (\ref{kap}))
\begin{align*}
H_{k}(x,0)=\int_{1}^{x}H_{-1}(s)ds - \alpha.
\end{align*}

To verify Corollary~\ref{nlog} let $F(x,k):=F_{k}(x)$. Let $F_{k}(x,k)$ denotes the partial derivative in $k$ of $F(x,k)$.  If we send $p \to 0, p<0$ in (\ref{our}) and compare the terms of order $p$ we obtain 
\begin{align*}
\int_{\mathbb{R}^{n}}\left(  \ln f - F_{k}\left(\frac{|\nabla f|}{f}, -1\right) \right)d\gamma_{n} \geq \ln \left(\int_{\mathbb{R}^{n}} f d\gamma_{n} \right)
\end{align*}
It remains to find the function $F_{k}(x,-1)$. Let us equate terms of order $(k+1)$ as $k \to -1, k<-1$ in the following equality 
\begin{align*}
F\left(\frac{H_{x}(x,k)}{H(x,k)},k \right) = \frac{H(x,k+1)}{H(x,k)^{1+\frac{1}{k}}}. 
\end{align*}
The straightforward computation shows that 
\begin{align*}
F_{k}\left( \frac{H_{-2}(x)} {H_{-1}(x)},-1\right) = H_{k}(x,0)+\ln H_{-1}(x)=\int_{1}^{x}H_{-1}(s)ds - \alpha + \ln H_{-1}(x)
\end{align*}
where 
\begin{align*}
\alpha=\int_{1}^{\infty}\left( H_{-1}(s) - \frac{1}{s}\right)ds.
\end{align*}

 \section{Concluding remarks}
 
 The reader may wander how we guessed the choice (\ref{f1}). Of course it was not a random guess. Function (\ref{f1}) is the best possible in the sense that the determinant of (\ref{matrica}) is identically zero
\begin{align}\label{monge}
&M_{yy}(M_{xx}+\frac{M_{y}}{y})-M_{xy}^{2}=0,\\
&M(x,0)=x^{p} \quad \text{for} \quad   x \geq 0. \nonumber
\end{align}
 Initially this was the way we started looking for  $M(x,y)$ as the solution of the  Monge--Amp\`ere equation with a drift (\ref{monge}). By a proper change of variables the equation reduces to the backwards heat equation  (see \cite{IV} for more details where the connection with R.~Bryant, Ph.~Griffiths theory of exterior differential systems was exploited)
 \begin{align}
 &u_{xx}+u_{t}=0, \label{h1}\\
 &u(x,0)=C x^{\frac{p}{p-1}} \label{h2} \quad \text{for} \quad  x\geq 0.
 \end{align}
One can notice that the Hermite polynomials  do satisfy (\ref{h1}) and (\ref{h2}) when $\frac{p}{p-1}$ is a positive integer. In general, one should invoke Hermite functions 
 and this is the reason of appearance of these functions in our theorem.  
 
 Another possibility is to assume that $M(x,y)$ should be homogeneous of degree $p$ which enforces $M$ to have the form (\ref{k1}) for some $F$. Next setting $h=\frac{F}{F'}$ 
 and   further by a  subtle  change of variables 
 one obtains  Hermite differential equation (\ref{hermite}).
 
Nevertheless, for the formal proof of Theorem~\ref{better}  we do not need to go through the details. We have  $M(x,y)$ defined by (\ref{f1}) and we just need to check that it satisfies the desired properties.

 The fact that $M(x,y)$  (see (\ref{f1})) satisfies (\ref{matrica}) makes it possible to extend Theorem~\ref{better}  in a semigroup setting  for uniformly log-concave measures. Indeed, let  $d\mu = e^{-U}dx$ where $\mathrm{Hess}\, U \geq R\cdot Id,$ $R>0$. Let $L=\Delta - \nabla U \cdot \nabla$, and let $P_{t} = e^{tL}$ be the semigroup with generator $L$ (see \cite{IV, BGL}).  
 \begin{corollary}
 For any $p \in \mathbb{R}\setminus [0,1]$ and  any smooth bounded $f\geq0$ with $\int_{\mathbb{R}^{n}} f^{p} d\mu <\infty$  we have 
 \begin{align*}
 P_{t} \left[f^{p} F_{\frac{1}{p-1}}\left(\frac{| \nabla f|}{f\sqrt{R}}\right) \right] \leq (P_{t} f)^{p} F_{\frac{1}{p-1}}\left(\frac{| \nabla P_{t} f|}{P_{t} f \sqrt{R}}\right).
 \end{align*}
 The inequality is reversed if $p\in (0,1)$.
 \end{corollary}
 \begin{proof}
 Notice that $\tilde{M}(x,y) = M(x,\frac{y}{\sqrt{R}})$ satisfies (\ref{matrica}). Now it remains to use inequality (2.3) from \cite{IV}. 
 \end{proof}
 Next by taking $t \to \infty$ and using the fact that $| \nabla P_{t} f | \leq e^{-tR} P_{t} | \nabla f |$ we obtain the following corollary 
 \begin{corollary}
  Let $d\mu =e^{-U}dx$ where $\mathrm{Hess}\, U \geq R \cdot Id$ for some $R>0$.
  For any $p \in \mathbb{R}\setminus [0,1]$ and  any smooth bounded $f \geq  0$ with $\int_{\mathbb{R}^{n}} f^{p} d\mu<\infty$ we have 
 \begin{align*}
 \int_{\mathbb{R}^{n}} f^{p} F_{\frac{1}{p-1}}\left(\frac{| \nabla f|}{f\sqrt{R}}\right) d\mu  \leq  \left( \int_{\mathbb{R}^{n}} f d\mu \right)^{p}.
  \end{align*}
  The inequality is reversed if $p \in (0,1)$. 
 \end{corollary}
 \begin{proof}
 See Corollary~1 in \cite{IV}.
 \end{proof}
  The limiting cases of these inequalities when  $p \to \pm \infty$ and  $p \to 0$   should be understood in the sense of functions $M_{\exp}$ and $M_{-\ln}$ as in Corollary~\ref{e-sob} and Corollary~\ref{nlog}. 
  
 Finally we would like to mention that having characterization (\ref{matrica}) of functional inequalities (\ref{ivo1}) makes approach to the problem (\ref{our}) systematic.   Very similar \emph{local estimates}  happen to rule some \emph{global inequalities}. We refer the reader to our recent papers on this subject \cite{IV01, IV02, IV03}.

 \subsection*{Acknowledgements} We are very grateful to Robert Bryant who suggested a change of variables in (\ref{determinant}). 

\end{document}